\documentclass[12pt]{amsart}
\usepackage{amssymb,amsbsy,amsmath,amsfonts,amssymb,amscd}
\usepackage{latexsym}
\usepackage{graphics}
\usepackage{color}

\input xy
\xyoption{all}

\newcommand\De{\Delta}

\DeclareMathOperator{\Alb}{Alb}

\newcommand{\CC}{\ensuremath{\mathbb{C}}}

\newcommand{\ZZ}{\ensuremath{\mathbb{Z}}}

\newcommand{\hol}{\ensuremath{\mathcal{O}}}

\newcommand{\PP}{\ensuremath{\mathbb{P}}}

\newcommand{\ra}{\ensuremath{\rightarrow}}

\def\eea{\end{eqnarray*}}
\def\bea{\begin{eqnarray*}}

\newcommand\dual{\mathrel{\raise3pt\hbox{$\underline{\mathrm{\thinspace d
\thinspace}}$}}}

\newcommand\QED{\ifhmode\unskip\nobreak\fi\quad {\rm Q.E.D.}} 
\newcommand\qe{\ifhmode\unskip\nobreak\fi\quad $\Box$}       

\def\BOX{\hfill\lower.5\baselineskip\hbox{$\Box$}}

\newtheorem{theo}{Theorem}[section]
\newtheorem{remarkk}[theo]{Remark}
\newenvironment{rem}{\begin{remarkk}\rm}{\end{remarkk}}

\newtheorem{defin}[theo]{Definition}
\newenvironment{definition}{\begin{defin}\rm}{\end{defin}}

\newtheorem{prop}[theo] {Proposition}
\newtheorem{cor}[theo]{Corollary}
\newtheorem{lemma}[theo]{Lemma}
\newtheorem{example}[theo]{Example}

\newtheorem{conj}[theo]{Conjecture}

\DeclareMathOperator{\Aut}{Aut}

\title[Bloch conjecture for Inoue surfaces ]{Bloch's conjecture for Inoue surfaces with $p_g=0$, $K^2 = 7$.}

\author[I. Bauer]{I. Bauer}
 \thanks{The present work took place in the realm of the DFG
Forschergruppe 790 "Classification of algebraic surfaces and
compact complex manifolds".}
\date{\today}

\begin{document}
\maketitle

\section*{Introduction} 
Let $S$ be a smooth projective complex surface and let 
$$
A_0(S) = \bigoplus_{i=0}^{\infty} A_0^i(S)
$$
be the group of rational equivalence classes of zero cycles on $S$. Then {\em Bloch's conjecture} asserts the following:

\begin{conj}[S. Bloch, \cite{bloch}]  
Let $S$ be a smooth surface with $p_g(S) =0$. Then the kernel $T(S)$ of the natural morphism
$$
A_0^0(S) \rightarrow \Alb(S)
$$
is trivial.
\end{conj}

The conjecture has been proven for surfaces $S$ with $\kappa(S) <2$ by Bloch, Kas and Liebermann (cf. \cite{bkl}), and has been verified for several examples (cf. e.g. \cite{barlow}, \cite{coughlanchan}, \cite{inose}, \cite{keum}, \cite{voisin}). It has been observed recently that by a beautiful result due to S. Kimura (cf. \cite{kimura})  all product quotient surfaces (i.e., minimal models of $C_1 \times C_2 /G$, where $G$ is a finite group acting on the product of two curves of respective genera at least 2) with $p_g=0$ satisfy Bloch's conjecture (cf. \cite{4names}).

Even if nowadays a substantial number of examples are known to fulfill Bloch's conjecture, there is still no idea how to prove the result in general. Also worth mentioning is that to our knowledge Bloch's conjecture has not been verified for any fake projective plane, i.e., a surface of general type with $p_g= 0$ and $K_S^2=9$.

The main result of this note is to verify Bloch's conjecture for Inoue surfaces with $p_g=0$ and $K^2 = 7$. Inoue surfaces are up to now\footnote{In the meantime a new family of surfaces of general type with $K_S^2 =7$, $p_g=0$ has been constructed by Yifan Chen, cf. \cite{yifan}} the only known family with $p_g=0$ and $K_S^2 =7$. They form a four dimensional irreducible connected component $\mathfrak{N}_I$ in the Gieseker moduli space $\mathfrak{M}_{1,7}^{can}$ of  canonical models of surfaces of general type with $p_g=0$, $K^2=7$, as was shown among other things in \cite{bacainoue}.

These surfaces were first constructed by M. Inoue in \cite{inoue} as as quotients of  complete intersections of two divisors (explicitly given by equations) of respective multi-degrees $(2,2,2,0)$ and $(0,0,2,2)$ by a free $(\ZZ / 2 \ZZ)^5$-action.

They can also be described as bidouble covers of the 
four nodal cubic surface (cf.  \cite{mlp}). This description will be crucial for the proof of our main result.

\begin{theo}\label{main}
Let $S$ be an Inoue surface with $K_S^2 =7$ and $p_g=0$. Then
$$
T(S) = A_0^0(S) =0,
$$
i.e., $S$ satisfies Bloch's conjecture.
\end{theo}
 The proof will in fact use the method of "enough automorphisms" introduced by Inose and Mizukami (cf. \cite{inose}) and refined by Barlow (cf. \cite{barlow}), but in a much simplified form.
\begin{rem}
As shown in \cite{bacainoue}, if $[S] \in \mathfrak{N}_I$, then $S$ is an Inoue surface. Therefore our result shows Bloch's conjecture for {\em each  surface} in the irreducible connected component $\mathfrak{N}_I$. 
\end{rem}

\section{Bloch's conjecture for surfaces with a $(\ZZ / 2 \ZZ)^2$-action}

The aim of this note is to prove Bloch's conjecture for Inoue surfaces using the method of "enough automorphisms" introduced by Inose and Mizukami (cf. \cite{inose}), refined by Barlow (cf. \cite{barlow}).

We need the following notation.

\begin{defin}
Let $G$ be a finite group and $H \leq G$ be a subgroup. Then we set:
$$
z(H) := \sum_{h \in H} h \in \CC G.
$$
\end{defin}

We recall Barlow's reformulation of the criterion of Inose and Mizukami in \cite{inose}.

\begin{lemma}\label{autos}[Precise version of Inose's "enough automorphisms", \cite{barlow}]
Let $S$ be a nonsingular surface and $G$ a finite subgroup of $\Aut(S)$. Let $H$, $H_1, \ldots , H_r$ be subgroups of $G$. We denote by $\mathcal{I}$ the two-sided ideal of $\CC G$ generated by $z(H_1), \ldots , z(H_r)$. Assume that
\begin{enumerate}
\item $z(H) \in I$,
\item $T(S/H_i) = 0, \ \forall i \in \{1, \ldots \}$.
\end{enumerate}
Then $T(S/H) = 0$.
\end{lemma}

Using the above we can show the following
\begin{prop}\label{criterion1}
Let $S$ be a surface of general type with $p_g(S) =0$. Assume that $G = (\ZZ /2 \ZZ)^2  \leq \Aut(S)$. Then $S$ satisfies Bloch's conjecture if and only if for each $\sigma \in G \setminus \{0\}$ the quotient $S/ \sigma$ satisfies Bloch's conjecture.
\end{prop}

\begin{rem}
Note that $S/\sigma$ is a surface with at most nodes and denoting by $X_{\sigma} \rightarrow S/\sigma$ its resolution of singularities, $X_{\sigma}$ is minimal and has $p_g=0$. Moreover, since nodes are rational singularities, $T(S/\sigma) = T(X_{\sigma})$. 
\end{rem}

\begin{proof}
If $S$ satisfies Bloch's conjecture then obviously each quotient by an involution also.

For the other direction we apply lemma \ref{autos} for $G= (\ZZ /2\ZZ)^2$, $H=0$, $H_1 = \langle \gamma_1 \rangle$, $H_2 = \langle \gamma_2 \rangle$, $H_3 = \langle \gamma_3 \rangle$, where $\gamma_1, \gamma_2, \gamma_3 = \gamma_1 \gamma_2$ are the three non trivial elements of $G$.
Then by assumption $S/H_i$ satisfies Bloch's conjecture, i.e. $T(S/H_i) = T(X) = 0$. 

Therefore it remains to verify that $1 = z(H) \in \mathcal{I}$, where $\mathcal{I}$ is the ideal in $\CC G$ generated by $z(H_1), z(H_2), z(H_3)$.
Observe that $z(H_i) = 1 + (\gamma_i)_*$, i.e. $\gamma_i \equiv -1 \mod  \mathcal{I}$. On the other hand, $(\gamma_3)_* = (\gamma_1)_*(\gamma_2)_* \equiv 1 \mod \mathcal{I}$, whence $1 = z(H) \in \mathcal{I}$.

\end{proof}

\begin{cor}\label{criterion}
Let $S$ be a surface of general type with $p_g(S) =0$ and assume that $G = (\ZZ /2 \ZZ)^2  \leq \Aut(S)$. Assume that for each $\sigma \in G \setminus \{0\}$ the quotient $S/ \sigma$ has $\kappa(S/ \sigma) \leq 1$, then $S$ satisfies Bloch's conjecture, i.e., $T(S) = A_0^0(S) = 0$. 
\end{cor}

\begin{proof}
Follows immediately combining proposition \ref{criterion1} and  \cite{bkl}.
\end{proof}
\section{Inoue surfaces with $p_g=0$ and 
$K_S^2 = 7$ as bidouble covers of the four nodal cubic}

In \cite{inoue} the author constructs 
 a family of 
minimal surfaces of general type $S$ with $p_g= 0$, 
$K_S^2
= 7$ as  quotients of a complete intersection of two divisors (explicitly given by equations) of respective multi-degrees $(2,2,2,0)$ and $(0,0,2,2)$ by a free $(\ZZ / 2 \ZZ)^5$-action.

In order to prove our main result we use a  different description of the Inoue surfaces, given by Mendes Lopes and Pardini in \cite{mlp},  as $(\ZZ / 2 \ZZ)^2$ - Galois covers of the 
four nodal cubic. 

We briefly recall their construction  here, for 
details we refer to the original article
\cite{mlp}, example 4.1.

Let $\Lambda$ in $\PP^2$  be a complete 
quadrilateral and 
denote the vertices by $P_1, \ldots,
P_6$.

We have labeled the 
vertices in a way that
\begin{itemize}
  \item the intersection point 
of the line $\overline{P_1P_2}$ and the line 
$\overline{P_3P_4}$ is $P_5$,
\item the 
intersection point of $\overline{P_1P_4}$ and $\overline{P_2P_3}$ is 
$P_6$.
\end{itemize}

Let $Y \rightarrow \PP^2$ be the blow up in 
$P_1, \ldots , P_6$, 
denote by $L$  the total transform of a line in 
$\PP^2$, let $E_i$, $1
\leq i \leq6$, be the exceptional curve lying over 
$P_i$.
Moreover, we denote by $S_i$, $1 \leq i \leq 4$, the strict 
transforms on $Y$ of the sides $S_i : = \overline{P_i P_{i+1}}$ for $ 1\leq i \leq 3$,
$S_4 : = \overline{P_4 P_{1}}$, of the quadrilateral $\Lambda$.

We denote by $\Delta_i$, $1 \leq i \leq 3$, the strict transforms of the three diagonals of the complete quadrilateral on $Y$, i.e., 
\begin{itemize}
\item $\Delta_1 \equiv L - E_1 - E_3$,
\item $\Delta_2 \equiv L - E_2 - E_4$,
\item $\Delta_3 \equiv L - E_5 - E_6$.
\end{itemize}

Observe that the four (-2) curves $S_i$ come from the resolution
of the 4 nodes of the cubic surface $\Sigma$ which is the anticanonical image of $Y$,
and the curves $E_i$ are the strict transforms of the 6 lines in $\Sigma$ connecting pairs of nodal points.

The surface $\Sigma$ contains also a triangle of lines (joining the midpoints of opposite edges
of the tetrahedron with sides the lines corresponding to the curves $E_i$). These are the 3 strict transforms 
 $\Delta_1$, $\Delta_2$,
$\Delta_3$ of the three diagonals of the complete quadrilateral $\Lambda$. 

For each line $\De_i$ in the cubic surface  $\Sigma$ we consider the pencil of planes containing
them, and the base point free pencil of residual conics, which we denote by $| f_i|$. Hence we have
$$| f_i| = |(-K_Y) -  \De_i| , \ \  \De_i + f_i \equiv (-K_Y) .$$

In the plane realization we have:
\begin{itemize}
  \item $f_1$ is the strict 
transform on $Y$ of a general element of
the pencil of conics $\Gamma_1$ through
$P_2,P_4,P_5,P_6$,
\item $f_2$ is the strict transform on $Y$ of a general element of
the pencil of conics $\Gamma_2$ through $P_1,P_3,P_5,P_6$,
\item $f_3$ is the strict transform on $Y$ of a general element of
the pencil of conics $\Gamma_3$ through $P_1,P_2,P_3,P_4$.
\end{itemize}

It is then easy to see that each curve $S_h$ is disjoint from the other curves $S_j$ ($j\neq h$), $\De_i$, and $f_i$,
if $f_i$ is smooth. Moreover, 
$$\De_i \cdot f_i = 2,  \  \  \De_i \cdot f_j = 0, i \neq j, \ \ f_i ^2 = 0 , \ \  f_j f_i = 2,  i \neq j . $$

\begin{definition}\label{Inouedivisors}
   We define the {\em Inoue divisors} on $Y$ as follows:
\begin{itemize}
\item $D_1:= \Delta_1 + f_2 + S_1 + S_2$, where $f_2 \in |f_2|$ smooth;
\item $D_2:= \Delta_2 + f_3$, where $f_3 \in |f_3|$ smooth;
\item$D_3:= \Delta_3 + f_1 +f_1'+ S_3 + S_4$, where $f_1, f_1' \in
|f_1|$ smooth.
\end{itemize}

\end{definition} 

Let $\pi \colon \tilde{S} \rightarrow Y$ be the
bidouble covering with branch divisors $D_1, D_2, D_3$ (associated to the 
3 nontrivial elements $\gamma_1, \gamma_2, \gamma_3$ of the Galois group $G \cong (\ZZ /2 \ZZ)^2$).

 Then $\tilde{S}$ is smooth and by the previous remarks we see
that over each $S_i$ there are two disjoint $(-1)$-curves.
Contracting these eight exceptional curves we obtain a minimal surface
with $p_g=0$ and $K_S^2 = 7$. 

Moreover, $S$ is a smooth $(\ZZ/2 \ZZ)^2$-covering of the four nodal cubic $\Sigma$, obtained from $Y$ by contracting the four $(-2)$-curves $S_j$ and by \cite{mlp} these are exactly the Inoue surfaces.

\begin{rem}\label{ratvar}
We immediately see that there is an open dense subset  in the product 
$$
|f_1| \times |f'_1| \times |f_2| \times |f_3| \cong (\PP^1)^4
$$
parametrizing the family of Inoue surfaces.
\end{rem}

\begin{rem}
   Denoting by $\chi_i \in G^*$ the nontrivial character orthogonal to $\gamma_i$, the non trivial character sheaves of this bidouble cover are
\begin{itemize}
   \item $\mathcal{L}_1 = \hol_Y( - K_Y + f_1 - E_4)$;
\item $\mathcal{L}_2 = \hol_Y( - 2 K_Y  - E_5 - E_6)$;
\item $\mathcal{L}_3 = \hol_Y( - K_Y + L - E_1 - E_2  - E_3)$.
\end{itemize}
I.e., $G$ acts on $\mathcal{L}_i^{-1}$ via the character $\chi_i$.
\end{rem}

Then we have the following:

\begin{prop}\label{charbican}[Mendes Lopes, Pardini \cite{mlp}]
The decomposition of $H^0(\tilde{S},\hol_{\tilde{S}}(2K_{\tilde{S}}))$ as sum of  isotypical components is as follows:
$$
H^0(\tilde{S},\hol_{\tilde{S}}(2K_{\tilde{S}})) \cong H^0(\tilde{S},\hol_{\tilde{S}}(2K_{\tilde{S}}))^{G} \oplus \bigoplus_{i=1}^3 H^0(\tilde{S},\hol_{\tilde{S}}(2K_{\tilde{S}}))^{\chi_i}, where
$$
\begin{itemize}
\item[(0)] $H^0(\tilde{S},\hol_{\tilde{S}}(2K_{\tilde{S}}))^{G} = H^0(Y,\hol_Y(-K_Y + f_1 + \sum S_j)) \cong \CC^7$,
\item[(1)] $H^0(\tilde{S},\hol_{\tilde{S}}(2K_{\tilde{S}}))^{\chi_1} = H^0(Y,\hol_Y(-K_Y + f_1 + \sum S_j - \mathcal{L}_1)) \cong \CC$,
\item[(2)] $H^0(\tilde{S},\hol_{\tilde{S}}(2K_{\tilde{S}}))^{\chi_2} = H^0(Y,\hol_Y(-K_Y + f_1 + \sum S_j- \mathcal{L}_2)) =0$,
\item[(3)] $H^0(\tilde{S},\hol_{\tilde{S}}(2K_{\tilde{S}}))^{\chi_3} = H^0(Y,\hol_Y(-K_Y + f_1 + \sum S_j- \mathcal{L}_3)) =0$.
\end{itemize}
\end{prop}

\begin{proof}
See \cite{mlp}, example 4.1., p.271.
\end{proof}

\begin{cor}\label{dimchar}
The decomposition of $H^0(S, \hol_S(2K_S))$ in invariant and anti-invariant part with respect to $\gamma_i$ is as follows:
$$
H^0(S, \hol_S(2K_S))^{+,\gamma_i} \cong H^0(\tilde{S},\hol_{\tilde{S}}(2K_{\tilde{S}}))^{G} \oplus H^0(\tilde{S},\hol_{\tilde{S}}(2K_{\tilde{S}}))^{\chi_i},
$$
$$
H^0(S, \hol_S(2K_S))^{-,\gamma_i} \cong H^0(\tilde{S},\hol_{\tilde{S}}(2K_{\tilde{S}}))^{\chi_j} \oplus H^0(\tilde{S},\hol_{\tilde{S}}(2K_{\tilde{S}}))^{\chi_k},
$$
where $\{i,j,k\} = \{1,2,3\}$.

In particular, $h^0(S, \hol_S(2K_S))^{-,\gamma_1} = 0$  and $h^0(S, \hol_S(2K_S))^{-,\gamma_i} = 1$  for $i=2,3$.
\end{cor}

\section{Quotients of Inoue surfaces by an involution}

In order to prove  theorem \ref{main} we have to show that for each automorphism $\sigma \in G= (\ZZ / 2 \ZZ)^2$ of an Inoue surface $S$ we have
$$
\kappa(S/ \langle \sigma \rangle) \leq 1.
$$

Before doing this, we have to fix some notation. 

\noindent
Let $S$ be a minimal regular surface of general type and let $\sigma$ be an involution on $S$. Then $\sigma$ is biregular, and its fixed locus consists of $k$ isolated points and a nonsingular (not necessarily connected) curve $R$. The quotient $T:=S/\langle \sigma \rangle$ has $k$ nodes, and resolving them we get a cartesian diagram of morphisms

\begin{equation}\label{diagram}
\xymatrix{
\hat{S}\ar_{\hat{\pi}}[d]\ar^{\epsilon}[r]&S\ar_{\pi}[d]\\
\hat{T}\ar[r]&T
}
\end{equation}
with vertical maps finite of degree $2$ and horizontal maps birational. 

We denote by $\Delta$ the branch curve $\pi(R)$ and by $E_1,\ldots,E_k$ the exceptional curves of $\epsilon$.

The action of $\sigma$ on $\hat{S}$ yields a decomposition
$\hat{\pi}_* \hol_{\hat{S}}=\hol_{\hat{T}} \oplus \hol_{\hat{T}}(-\hat{\delta})$, with $2\hat{\delta} \equiv \Delta +\sum_1^k \hat{\pi}(E_i)$. Recall that $K_{\hat{S}}\equiv \hat{\pi}^*(K_{\hat{T}}+\hat{\delta})$.

Then (cf.  \cite{mlp2}, \cite{borrelli}):

\begin{lemma}
\begin{equation}\label{rito}
0\leq k= K_S^2 + 6\chi(\hol_{\hat{T}}) - 2\chi(\hol_S) -
2h^0(\hol_{\hat{T}}(2K_{\hat{T}} + \hat{\delta})).
\end{equation} 
Moreover, if $p_g(\hat{T}) = 0$, then the biconical map of $S$ factors through $\sigma$ if and only if $h^0(\hol_{\hat{T}}(2K_{\hat{T}} + \hat{\delta})) = 0$.
\end{lemma}

Combining the above lemma with corollary \ref{dimchar} we obtain:

\begin{prop}\label{kforinoue}
Let $S$ be an Inoue surface with $p_g(S) = 0$ and $K_S^2 = 7$ and let $\gamma_i$ be one of the nontrivial elements of $G \cong (\ZZ / 2 \ZZ)^2$. Then we have for the number of isolated fixed points $k_i$ of $\gamma_i$:
\begin{itemize}
\item $k_1 = 11$, in particular the biconical map factors through $\gamma_1$;
\item $k_2 = k_3 = 9$.
\end{itemize}
\end{prop}

\begin{proof}
Note that $k_i = K_S^2 + 6\chi(\hol_{\hat{T}}) - 2\chi(\hol_S) -
2h^0(\hol_{\hat{T}}(2K_{\hat{T}} + \hat{\delta})) = 7 +4 -2h^0(\hol_{\hat{T}}(2K_{\hat{T}} + \hat{\delta}))$. Now use 
$$
h^0(\hol_{\hat{T}}(2K_{\hat{T}} + \hat{\delta})) = h^0(\hat{S}, \hol_{\hat{S}}(2K_{\hat{S}}))^{-,\gamma_i} =h^0(S, \hol_S(2K_S))^{-,\gamma_i}.
$$
The claim follows now from corollary \ref{dimchar}.
\end{proof}

We also need the following results of \cite{dmlp}:

\begin{lemma}\label{formulas}[\cite{dmlp}, lemma 4.2.]
Let $S$ be a smooth surface with $p_g(X) = q(X) =0$ and let $\sigma$ be an automorphism of $S$ of order 2. We denote again the divisorial part of the fixed locus of $\sigma$ by $R$ and by $k$ the number of isolated fixed points. Moreover, let $t$ be the trace of $\sigma | H^2(S, \CC)$. Then:
$$
k = K_S \cdot R +4, \ \ t=2-R^2.
$$
Furthermore, using the notation of diagram \ref{diagram} we have the following relation for the Picard numbers:
$$
\rho(S) + t = 2 \rho(\hat{T}) -2k.
$$
\end{lemma}

\begin{prop}\label{rho}[\cite{dmlp}, prop.4.1.]
Let $Y$ be a surface with $p_g(Y) = q(Y) =0$ and Kodaira dimension $\kappa(Y) \geq 0$. Moreover, let $C_1, \ldots , C_k \subset Y$ be disjoint rational $(-2)$-curves. Then:
\begin{itemize}
\item[(i)] $k \leq \rho(Y)-2$;
\item[(ii)] if $k= \rho(Y) -2$, then $Y$ is minimal.
\end{itemize}
\end{prop}

In fact, we need also to consider the case $\rho(Y) = k-3$. Here $Y$ is not necessarily minimal, but using the same line of arguments as in prop. 4.1. of \cite{dmlp} we can show the following:

\begin{lemma}\label{k=r-3}
If in proposition \ref{rho} we have $k= \rho(Y) -3$, then the minimal model $\bar{Y}$ of $Y$ is either
\begin{itemize}
\item[-] equal to Y, or
\item[-] $Y \ra \bar{Y}$ is the blow up in one point, in particular $K_Y^2 = K_{\bar{Y}}^2 - 1$, or
\item[-] $Y \ra \bar{Y}$ is the blow up in two infinitely near points,  in particular $K_Y^2 = K_{\bar{Y}}^2 - 2$.
\end{itemize}
\end{lemma}

\begin{proof}
Assume that $Y$ is not minimal. Let $E \subset Y$ be an irreducible $(-1)$ curve and let $Y'$ be the surface obtained by blowing-down $E$.

If $E$ does not intersect any of the nodal curves $C_i$, then $Y'$ contains $k$ disjoint nodal curves, hence $k = \rho(Y)-3 = \rho(Y')-2$. Therefore, by proposition \ref{rho}, $Y'$ is minimal.

Assume now that $E$ intersects, say $C_1$, i.e., $E \cdot C_1 = \alpha >0$. Then, if $C_1'$ denotes the image of $C_1$ in $Y'$, $C_1'$ is irreducible and
$$
(C_1')^2 = -2 + \alpha^2, \ \ C_1' \cdot K_{Y'} = - \alpha.
$$
Suppose $\alpha \geq 2$, then $(C_1')^2 >0$ and the image $C_1''$ in $\bar{Y}$ is a curve and satisfies $C_1''\cdot K_{\bar{Y}} \leq C_1' \cdot K_{Y'} < 0$. This contradicts $K_{\bar{Y}}$ nef.

This implies that $\alpha = 1$, i.e., $C_1'$ is a $(-1)$-curve. Moreover, $E \cdot C_i=0$ for $i \neq 1$, since otherwise we would have on $Y'$ two intersecting $(-1)$-curves, which is not possible on a surface with $\kappa(Y) \geq 0$. 
Denote by $Y''$ the surface obtained from $Y'$ by blowing down $C_1'$. Then $Y''$ contains $k-1$ nodal curves and 
$$
k-1 = \rho(Y)-4 = \rho(Y'')-2.
$$
By proposition \ref{rho} we get that $Y''$ is minimal.

\end{proof}

Now we are ready to prove our main result.
\begin{proof}[Proof of theorem \ref{main}.]
We will in fact show the following more general result
\begin{prop}\label{quotients}
Let $S$ be a minimal surface of general type with $p_g=0$ and $K_S^2 = 7$. Let $\sigma$ be an involution on $S$ with $k$ isolated fixed points. If $\kappa(S/ \langle \sigma \rangle) = 2$, then $k= 5$ or $k = 7$.
\end{prop}

\begin{proof}[Proof of prop. \ref{quotients}]
Since $S$ is a regular surface with $p_g=q=0$, $\rho(S) = e(S) - 2$. Therefore, $K_S^2 =7$ implies $\rho(S) = 3$. 

Since the class of the canonical divisor in $H^2(S,\CC)$ is invariant under $\sigma$ we have for $t$ the possibilities $t=3,1$ or $-1$.

Assume that $t=-1$, i.e. $\sigma|H^2(S, \CC)$ has eigenvalues $1,-1,-1$ and in particular, $\dim H^2(S, \CC)^{inv} = 1$. This implies that  $K_S $ is numerically equivalent to $rR$, which contradicts $R^2 = 2-t = 3$. Therefore this case does not occur.

\noindent
\underline{$t=1$}: then $R^2 =1$ and $K_S \cdot R = 2m+1$, for some $m \geq 0$. This implies:
$$
k= K_S \cdot R + 4 = 5 + 2m \geq 5.
$$
On the other hand, by lemma \ref{formulas},  
$$4 = \rho(S) + 1 = 2 \rho(\hat{T}) - 2k,$$ 
whence $\rho(\hat{T}) =k+2$.

Since $\kappa(S/\langle \sigma \rangle) = 2$, it follows by prop. \ref{rho} that $\hat{T}$ is minimal. In particular, $K_{\hat{T}}^2 > 0$. Therefore,
$$
K_{\hat{T}}^2 = 12 - e(S) = 8 - k > 0.
$$
This implies that $k = 5$ or $k = 7$.

\noindent
\underline{$t=3$}: in this case $R^2 = -1$, whence again $K_S \cdot R = 2m+1$, for some $m \geq 0$. This implies:
$$
k= K_S \cdot R + 4 = 5 + 2m \geq 5.
$$
On the other hand, by lemma \ref{formulas},   $\rho(\hat{T}) =k+3$.

Since $\kappa(S/\langle \sigma \rangle) = 2$, it follows by lemma. \ref{k=r-3} that  $K_{\hat{T}}^2 \geq -1$. Therefore,
$$
K_{\hat{T}}^2 = 12 - e(S) = 7 - k \geq -1.
$$
This immediately implies that $k= 5$ or $k=7$.
\end{proof}
Combining proposition \ref{quotients} with proposition \ref{kforinoue} we see that for an Inoue surface we have $k(S/\langle \gamma_i \rangle) \leq 1$. By  corollary \ref{criterion} our main result is proven.
\end{proof}

\begin{rem}
We were kindly informed that Proposition \ref{quotients} follows also from \cite{yongnam}, cf. e.g. the table on page 2, loc. cit.. 
\end{rem}


\bigskip
\noindent {\bf Author's Adress:}

\noindent I.Bauer \\
Mathematisches Institut der Universit\"at Bayreuth\\ NW II\\
Universit\"atsstr. 30\\ 95447 Bayreuth

\begin{verbatim}
ingrid.bauer@uni-bayreuth.de,
\end{verbatim}

\begin{thebibliography}{Grif-VW}
\bibitem[Bar85b]{barlow}
R. Barlow, 
{\em Rational equivalence of zero cycles for some more surfaces with $p\sb
g=0$}.
Invent. Math.  \textbf{79}  (1985),  no. 2, 303--308.

\bibitem[BCGP09]{4names} Bauer, I., Catanese, F., Grunewald, F.,
Pignatelli, R. {\it Quotients  of a product of curves by a finite
    group and their fundamental groups.} arXiv:0809.3420, to appear in
Amer. J. Math.

\bibitem[BC12]{bacainoue} Bauer, I., Catanese, F., {\it Inoue type manifolds and Inoue surfaces: a connected component of the moduli space of surfaces with $K^2 = 7$, $p_g=0$.} arXiv:1205.7042
 
\bibitem[Blo75]{bloch}
S. Bloch,
{\it $K\sb{2}$ of Artinian $Q$-algebras, with application to algebraic cycles}.
Comm. Algebra  {\bf 3}  (1975), 405--428.

\bibitem[BKL76]{bkl}
S. Bloch,  A. Kas, D. Lieberman, 
{\em Zero cycles on surfaces with $p\sb{g}=0$}.
Compositio Math.  {\bf 33}  (1976), no. 2, 135--145.

\bibitem[Bor07]{borrelli} Borrelli, G. {\em The classification of surfaces of general type with nonbirational bicanonical map. J. Algebraic Geom. 16 (2007), no. 4, 625Ð669.}

\bibitem[ChCou10]{coughlanchan}
Chan Mario T.,  Coughlan 
S.,  {\it  Kulikov surfaces form a 
connected component of the 
moduli space},
arXiv:1011.5574, to appear on Nagoya Math. Journal. 

\bibitem[Che12]{yifan}{\em A New Family of Surfaces with $K^2=7$ and $p_g=0$.}  arXiv:1210.4633

\bibitem[DMLP02] {dmlp} Dolgachev, I.; Mendes Lopes, M.; Pardini, R. {\em Rational surfaces with many nodes.} 
Compositio Math. 132 (2002), no. 3, 349Ð363.


\bibitem[IM79]{inose}
H. Inose, M. Mizukami, 
{\em Rational equivalence of $0$-cycles on some surfaces of general type with $p\sb{g}=0$}.  
Math. Ann.  {\bf 244} (1979), no. 3, 205--217.
 
\bibitem[In94]{inoue} Inoue, M.  {\em Some new 
surfaces of general 
type.} Tokyo J. Math. 17 (1994), no. 2, 
295--319.

\bibitem[Ke88]{keum} Keum, J.H. {\it Some new surfaces of 
general 
type with $p_g = 0$.}  Unpublished manuscript 
(1988).

\bibitem[Kim05]{kimura}
S. Kimura, 
{\it Chow groups are finite dimensional, in some sense}.
Math. Ann.  {\bf 331} (2005),  no. 1, 173--201.

\bibitem[LS10]{yongnam} Lee, Y., Shin Y.J., {\em Involutions on a surface of general type with $p_g=q=0$, $K^2=7$.} 
arXiv:1003.3595 to appear in Osaka J. Math.

\bibitem[ML-P01]{mlp} Mendes 
Lopes, M.; Pardini, R. {\it The 
bicanonical map of surfaces with 
$p_g=0$ and $K^2\geq 7$.} Bull.
London Math. Soc. 33 (2001), no. 3, 
265--274. 


\bibitem[ML-P03]{mlp2} Mendes 
Lopes, M.; Pardini, R. {\em The bicanonical map of surfaces with $p_g=0$ and $K^2 \geq 7$. II.} Bull. London Math. Soc. 35 (2003), no. 3, 337Ð343.  


\bibitem[Nai94]{naie94} Naie, D. {\it Surfaces 
d'Enriques et une 
construction de surfaces de type g\'en\'eral avec 
$p\sb g=0$}.
Math. Z.  215  (1994),  no. 2, 269--280.

\bibitem[Voi92]{voisin}
C. Voisin, 
{\em Sur les z\'ero-cycles de certaines hypersurfaces munies d'un automorphisme}. 
 Ann. Scuola Norm. Sup. Pisa Cl. Sci. (4) {\bf 19}  (1992),  no. 4, 473--492.


\end{thebibliography}
\end{document}